\documentclass{article}
\usepackage{amsfonts}
\usepackage{amsmath}
\usepackage{amsthm}
\usepackage[margin=1.0in]{geometry}
\usepackage{hyperref}
\usepackage{url}

\def\@begintheorem#1#2{\par\bgroup{\sc #1\ #2. }\it\ignorespaces}
\def\@opargbegintheorem#1#2#3{\par\bgroup{\sc #1\ #2\ (#3). } \it\ignorespaces}
\def\@endtheorem{\egroup}
\newtheoremstyle{slanttext}{8pt}{3pt}{\sl}{}{\bf}{.}{.5em}{}
\theoremstyle{slanttext}
\newtheorem{theorem}{Theorem}[section]
\newtheorem{corollary}[theorem]{Corollary}
\newtheorem{lemma}[theorem]{Lemma}

\begin{document}

\title{Tail diameter upper bounds for polytopes and polyhedra}
\author{J. Mackenzie Gallagher, {\sl George Mason University} \and Edward D. Kim, {\sl University of Wisconsin-La Crosse}}

\maketitle

\begin{abstract}
In 1992, Kalai and Kleitman proved a quasipolynomial upper bound on the diameters of convex polyhedra. Todd and Sukegawa-Kitahara proved tail-quasipolynomial bounds on the diameters of polyhedra. These tail bounds apply when the number of facets is greater than a certain function of the dimension. We prove tail-quasipolynomial bounds on the diameters of polytopes and normal simplicial complexes. We also prove tail-polynomial upper bounds on the diameters of polyhedra.
\end{abstract}

\section{Introduction}

Diameter upper bounds for convex polyhedra provide lower bounds to the number of simplex iterations needed for linear programming. Upper bounds for polytopes are generally similar, and are of interest in their own right due to Santos' counterexample (see~\cite{Santos:CounterexampleHirsch}) to the Hirsch Conjecture. Recall that a polyhedron is the intersection of a finite number of halfspaces and a polytope is the convex hull of a finite set of points, thus a polytope is a bounded polyhedron. For a given polytope or polyhedron $P$, we use $d$ to denote the dimension of its affine span and $n$ to denote the number of $(d-1)$-dimensional faces. The Hirsch Conjecture asserted that the diameter of a polytope was at most $n-d$. The conjecture is false (see~\cite{MatschkeSantosWeibel:Width5Prismatoids}) for $\frac{n}{2}=d \geq 20$. History on the Hirsch Conjecture is given in~\cite{Klee:The-d-step-conjecture},~\cite{Klee:d-step}, or the survey~\cite{KimSantos:HirschSurvey}. Our terminology follows~\cite{Grunbaum:Polytopes} and~\cite{Ziegler:Lectures}.

The best known upper bounds on the diameters of convex polyhedra are the quasipolynomial bounds given by Kalai and Kleitman (see~\cite{Kalai:Quasi-polynomial}), with recent improvements by Todd (see~\cite{Todd:KalaiKleitman}) and Sukegawa and Kitahara (see~\cite{Sukegawa:KalaiKleitman}). For polytopes, there are several classical upper bounds on the diameters of polytopes which are linear in the number $n$ of facets of $P$ given by Barnette (see~\cite{Barnette} and~\cite{Barnette:UpperBound}) and Larman (see~\cite{Larman}).

Recent research studies the diameters of abstract polytopes and related objects, following Adler et al. in~\cite{Adler:AbstractPolytopesThesis}, \cite{Adler:LowerBounds}, \cite{AdlerDantzigMurty:AbstractPolytopes}, \cite{Adler:MaxDiamAbsPoly}, \cite{AdlerSaigal:LongPathsAbstract}, \cite{Kalai:DiameterHeight}, \cite{Lawrence}, and \cite{Murty:GraphAbstract}. Eisenbrand et al.~(see~\cite{Eisenbrand:Diameter-of-Polyhedra}) introduced connected layer families as a generalization of polyhedra. In~\cite{Eisenbrand:Diameter-of-Polyhedra}, Eisenbrand et al.{} proved that the quasipolynomial upper bound of Kalai and Kleitman in~\cite{Kalai:Quasi-polynomial} applied to connected layer families. In~\cite{Kim:PolyhedralGraphAbstractions}, Kim generalized connected layer families to arbitrary graphs, defining subset partition graphs satisfying dimension reduction. Very recently, combinatorial paths have been studied in normal simplicial complexes (see, e.g.,~\cite{Labbe:ExpCombSegments} and~\cite{Santos:RecentProgress}). Normal simplicial complexes are studied in~\cite{Bogart:SuperlinearSPG} and~\cite{Kim:PolyhedralGraphAbstractions}, which are subset partition graphs satisfying dimension reduction in the dual setting.

In Section~\ref{section:PreviousBounds}, we review previous diameter bounds. In Section~\ref{section:polyhedra}, we present new tail-quasipolynomial upper bounds for the diameters of polyhedra. In Section~\ref{section:polytopes}, we revisit Sukegawa and Kitahara's (see~\cite{Sukegawa:KalaiKleitman}) tail-quasipolynomial upper bound for the diameters of polytopes. In Section~\ref{section:complexes}, we apply the recent techniques of Todd and Sukegawa-Kitahara to obtain tail-quasipolynomial upper bounds for normal simplicial complexes. Section~\ref{section:iterated-recursion} applies iterated recursion to a classical inequality of Kalai and Kleitman (see~\cite{Kalai:Quasi-polynomial}) to obtain the first tail-polynomial bounds in Section~\ref{section:tail-polynomial}. Our main result is Theorem~\ref{theorem:almost-linear}, a tail-almost-linear upper bound on the diameter of convex polyhedra when the number of facets is large. As with~\cite{Sukegawa:KalaiKleitman}, our bounds are tail bounds in that the presented inequalities apply when the number $n$ of facets is greater than some function of the dimension $d$. We conclude with final remarks in Section~\ref{section:concluding-remarks}.

\section{Previous bounds}\label{section:PreviousBounds}

Let $\Delta_b(d,n)$ denote the maximum diameter among $d$-dimensional polytopes with $n$ facets for $n > d$. We use $\Delta_u(d,n)$ to denote the maximum diameter of $d$-dimensional polyhedra with $n$ facets when $n \geq d$. Trivially, $\Delta_b(d,n) \leq \Delta_u(d,n)$ when both quantities are defined.

In this article, we expand on the upper bound $ \Delta_u(d,n) \leq (n-d)^{\log_2(d-1)}$ if $n \geq d \geq 3$ by Sukegawa and Kitahara (see~\cite{Sukegawa:KalaiKleitman}). As with the result $\Delta_u(d,n) \leq (n-d)^{\log_2 d}$ if $n \geq d \geq 1$ by Todd in~\cite{Todd:KalaiKleitman}, their proof relies on the recurrence relation
\begin{equation}\label{equation:KalaiKleitman-Deltau}
\Delta_u(d,n) \leq \Delta_u(d-1,n-1)+2\Delta_u(d,\lfloor n/2\rfloor) + 2 \quad \text{if $n \geq 2d$}
\end{equation}
in~\cite{Kalai:Quasi-polynomial} by Kalai-Kleitman. Their proof also uses the bound $\Delta_b(d,n) \leq 2^{d-3}n$ in~\cite{Larman} by Larman which only applies to \emph{bounded} polyhedra, but the recent bound $\Delta_u(d,n) \leq 2^{d-3}n$ from~\cite{Labbe:ExpCombSegments} applies.

Results will rely on upper bounds on the diameter which are linear in $n$ for fixed dimension $d$. The following are upper bounds on the diameter of a \emph{bounded} polytope which are linear in fixed dimension. In 1969, Barnette (see~\cite{Barnette}) proved $\Delta_b(d,n) \leq 3^{d-2}n$ if $ n \geq d \geq 4$. In 1970, Larman (see~\cite{Larman}) improved this bound to $\Delta_b(d,n) \leq 2^{d-3}n$ if $ n \geq d \geq 3$. In 1974, Barnette (see~\cite{Barnette:UpperBound}) improved the bound again: though the abstract and final corollary of~\cite{Barnette:UpperBound} both state a bound of $\Delta_b(d,n) \leq \frac132^{d-3}(n-d+\frac52)$ for $n \geq d \geq 3$, the inequality does not hold, e.g., for the $3$-cube. In fact, the proof gives $\Delta_b(d,n) \leq \frac132^{d-2}(n-d+\frac52)$ if $ n \geq d \geq 3$, as stated in Theorem 1 of~\cite{Barnette:UpperBound}, which strictly improves Larman's bound if and only if $n+2d > 5$. All three of these results define their object of study as the convex hull of a finite set of points, thus the results apply to $\Delta_b(d,n)$.

It would be interesting to ascertain if any of the statements in~\cite{Barnette, Barnette:UpperBound, Larman} apply to $\Delta_u(d,n)$ with few changes to the proof. (The definition of removable edges in~\cite{Barnette:UpperBound} requires $3$-regularity, which does not apply to $3$-polyhedron graphs.) Until very recently, the only upper bound to $\Delta_u(d,n)$ linear in $n$ when $d$ is fixed was $\Delta_u(d,n) \leq \Delta_{CLF}(d,n) \leq 2^{d-1}n$ by Eisenbrand et al.{} in~\cite[Theorem 3.2]{Eisenbrand:Diameter-of-Polyhedra}, where $\Delta_{CLF}(d,n)$ denotes the maximum diameter of $d$-dimensional connected layer families with $n$ symbols. In 2015, Labb\'e, Manneville, and Santos (see Theorem 3.5 in~\cite{Labbe:ExpCombSegments}) proved a bound on simplicial complexes which implies $\Delta_u(d,n) \leq 2^{d-3}n$, matching Larman's upper bound for $\Delta_b(d,n)$ in~\cite{Larman}.

The proofs and expositions in the literature interchange the known upper bounds for polytopes and polyhedra. This type of inconsistency appears, e.g., in~\cite{KimSantos:HirschSurvey}, which states~\cite{Kalai:Quasi-polynomial} proves $\Delta_b(d,n) \leq n^{1+\log_2(d)}$, while true, the results of~\cite{Eisenbrand:Diameter-of-Polyhedra} and~\cite{Kalai:Quasi-polynomial} prove the stronger statement $\Delta_u(d,n) \leq n^{1+\log_2(d)}$. We note the Kalai-Kleitman bound~\eqref{equation:KalaiKleitman-Deltau} holds for $\Delta_b(d,n)$.
\begin{lemma}[Labb\'e, Manneville, Santos~\cite{Labbe:ExpCombSegments}]
If $n \geq 2d$, then
\begin{equation}\label{equation:KalaiKleitman-Deltab}
\Delta_b(d,n) \leq \Delta_b(d-1,n-1) + 2\Delta_b(d,\lfloor n/2\rfloor) + 2.
\end{equation}
\end{lemma}
\begin{proof}
The proof is given in the polar setting. Since $\Delta_b(d,n)$ is maximized among simplicial polytopes, let $P$ be a simplicial $d$-polytope with $n$ vertices. Let $F^+$ and $F^-$ be two facets of $P$. We want to bound the dual-graph distance between $F^+$ and $F^-$.

For each $i=0, 1, 2, \dots$ let $\mathcal{F}^+_i$ and $\mathcal{F}^-_i$ be the set of facets at distance at most $i$ from $F^+$ and $F^-$, respectively. In particular, $\mathcal{F}^+_0 = \{F^+\}$. Let $k$ be the smallest index for which $\mathcal{F}^+_k$ and $\mathcal{F}^-_k$ have a vertex in common. Let $v$ be such a vertex. Then:

The vertex figure of $P$ at $v$ is a $(d-1)$-polytope with at most $n-1$ vertices, and the star of $v$ shares a facet with both $\mathcal{F}^+_k$ and $\mathcal{F}^-_k$. Therefore, $\operatorname{dist}(F^+,F^-) \leq 2k + \Delta_b(d-1,n-1)$.

At least one of $\mathcal{F}^+_{k-1}$ or $\mathcal{F}^-_{k-1}$ has at most $n/2$ vertices, because otherwise they would have a vertex in common. Without loss of generality, suppose $\mathcal{F}^+_{k-1}$ has at most $n/2$ vertices. Then $P_+=\operatorname{conv}(\mathcal{F}^+_{k-1})$ is a $d$-polytope with at most $n/2$ vertices. Moreover, all paths of length $k-1$ from $F^+$ in $P_+$ are paths in $\mathcal{F}^+_{k-1}$ as well. Hence, $k-1 \le \Delta_b(d,n/2)$.
\end{proof}

Recent research has focused on combinatorial abstractions of polytopes and polyhedra, with the first results due to Eisenbrand, H\"ahnle, Razborov, and Rothvo\ss{} (see~\cite{Eisenbrand:Diameter-of-Polyhedra}). Eisenbrand et al.{} introduced connected layer families, a structure recording the vertex-facet incidences represented by sets of symbols, and abstracting the concept of distances in so-called layers, which is based on the polyhedral property that every face of a polyhedron has a connected graph. In~\cite{Eisenbrand:Diameter-of-Polyhedra}, Eisenbrand et al.{} prove that the diameters of $\frac{n}{4}$-dimensional connected layer families with $n$ symbols is in $\Omega(n^2/\log n)$. Bogart and Kim extend the $\Omega(n^2/\log n)$ result in~\cite{Bogart:SuperlinearSPG} to the case when a layer family also satisfies a combinatorial adjacency condition and a cardinality condition, coming from the polyhedral property that two simple vertices in a $d$-dimensional polyhedron are adjacent if and only if they share $d-1$ facets. Similar incidence structures were studied by Santos in~\cite{Santos:RecentProgress} and by Labb\'e et al. (see~\cite{Labbe:ExpCombSegments}) in the dual setting, where the connectedness property of layer families corresponds to the condition that a pure $(d-1)$-dimensional simplicial complex is normal. In~\cite{Labbe:ExpCombSegments}, Labb\'e, Manneville, and Santos prove that the diameter of $(d-1)$-dimensional pure and normal pseudomanifolds without boundary is at most $2^{d-3}n$.

As in~\cite{Sukegawa:KalaiKleitman} and~\cite{Todd:KalaiKleitman}, all logarithms are base $2$ and we readily use the identity $x^{\log(y)}=y^{\log(x)}$ for positive reals $x,y$.

\section{Upper bounds for diameters of polyhedra}\label{section:polyhedra}

Following~\cite{Sukegawa:KalaiKleitman} by Sukegawa and Kitahara, we recursively define a doubly-indexed sequence $\tilde\Delta_u(d,n)$ for $n \geq d \geq 3$ by the following:
\begin{equation*}
\tilde\Delta_u(d,n) = \left\{
\begin{array}{ll}
n-3 & \text{if } d=3\\
0 &\text{if } n=d > 3\\
\tilde\Delta_u(d-1,n-1) &\text{if } 3 < d< n < 2d\\
\tilde\Delta_u(d-1,n-1)+2 \tilde\Delta_u(d,\lfloor n/2\rfloor)+2  &\text{if } 6 < 2d \leq n.
\end{array}
\right.
\end{equation*}
Since $\Delta_u(3,n) \leq n-3$ by~\cite[Theorem 2.4]{Klee:PathsII}, the recursion~\eqref{equation:KalaiKleitman-Deltau} implies $\Delta_u(d,n) \leq \tilde\Delta_u(d,n)$ for $n \geq d \geq 3$.

The result below is due to Sukegawa and Kitahara (see~\cite{Sukegawa:KalaiKleitman}), which improves the bound of Todd (see~\cite{Todd:KalaiKleitman}), though we present a different proof here to preview our results in Sections~\ref{section:polytopes} and~\ref{section:complexes}.
\begin{theorem}\label{theorem:DeltaU-k0}
If $n \geq d \geq 3$, then $\Delta_u(d,n) \leq (n-d)^{\log (d-1)}$.
\end{theorem}
\begin{proof}
Either $d \in \{3,4\}$ or $d \geq 5$. If $d=3$, then the stated bound of $(n-3)^{\log(3-1)}$ aligns with the known bound of $\Delta_u(3,n)\leq n-3$ in~\cite{Klee:PathsII} for all $n \geq 3$. If $d=4$, then Appendix~\ref{appendix:DeltaU-k0} shows $\tilde\Delta_u(4,n) \leq (n-4)^{\log (4-1)}$ for $4 \leq n \leq 45$ and~\cite{Eisenbrand:Diameter-of-Polyhedra} implies $\Delta_u(4,n) \leq 2^{4-1}n \leq (n-4)^{\log(4-1)}$ for $n \geq 46$.

Suppose $d \geq 5$. The cases when $n < 2d$ follow from monotonicity, so we assume $n \geq 2d$. Appendix~\ref{appendix:DeltaU-k0} also checks $\tilde\Delta_u(d,n) \leq (n-d)^{\log (d-1)}$ holds in the cases when $2d \leq n < d+8$. Thus, we assume $d \geq 5$ and $n-d\geq 8$, which implies $\log(n-d)\geq 3$. Then,
\begin{align*}
\Delta_u(d,n) &\leq \Delta_u(d-1,n-1) + 2 \cdot \Delta_u(d,\lfloor n/2 \rfloor) + 2 \\
&\textstyle \leq (n-d)^{\log(d-2)} + 2 \cdot (\frac{n}{2} - d)^{\log (d-1)} + 2\\
&= (d-2)^{\log(n-d)} + 2  \cdot (d-1)^{\log(n/2-d)} + 2\\
&=  \left(\frac{d-2}{d-1}\right)^{\log(n-d)} (d-1)^{\log(n-d)} + 2 \cdot (d-1)^{\log(n/2-d)} + 2\\
&\leq \left(\frac{d-2}{d-1}\right)^{\log(n-d)} (d-1)^{\log(n-d)} + 2(d-1)^{\log((n-d)/2)} + 2 \\
&= \left(\frac{d-2}{d-1}\right)^{\log(n-d)} (d-1)^{\log(n-d)} + \frac2{d-1} \cdot (d-1)^{\log(n-d)} + 2\\
&\leq \left(\frac{d-2}{d-1}\right)^{3} (d-1)^{\log(n-d)} + \frac2{d-1} \cdot (d-1)^{\log(n-d)} + 2\\
&= (d-1)^{\log(n-d)} \left[ \left(\frac{d-2}{d-1}\right)^{3}  + \frac2{d-1}   + \frac{2}{(d-1)^{\log(n-d)}} \right]\\
&\leq (n-d)^{\log (d-1)},
\end{align*}
which proves the result by induction on $n$.
\end{proof}
As a corollary, this easily implies that H\"ahnle's conjectured upper bound of $(n-1)d$ in~\cite{Hahnle:Diplomathesis} holds for $n-d \leq 4$, since $(n-d)^{\log(d-1)}\leq 4^{\log(d-1)}=(d-1)^2\leq (n-1)d$. Sukegawa and Kitahara (see~\cite{Sukegawa:KalaiKleitman}) also proved that if $d = 4$ and $n \geq 9$, or $d \geq 5$ and $n \geq d+3$, then $\Delta_u(d,n)\leq (n-d-1)^{\log(d-1)}$. As with Theorem~\ref{theorem:DeltaU-k0}, this result is a tail bound.

\section{Upper bounds for diameters of polytopes}\label{section:polytopes}

In this section, we prove an upper bound on $\Delta_b(d,n)$ using inequality~\eqref{equation:KalaiKleitman-Deltab} and Theorem 2.1 in~\cite{KimSantos:HirschSurvey}, which gives the tight bound $\Delta_b(3,n) \leq \lfloor\frac{2n}{3}\rfloor-1$, originally due to Klee in~\cite{Klee:PathsII}.

We recursively define a doubly-indexed sequence $\tilde\Delta_b(d,n)$ for $n \geq d \geq 3$ by the following:
\begin{equation*}
\tilde\Delta_b(d,n) = \left\{
\begin{array}{ll}
\lfloor\frac{2n}{3}\rfloor-1 & \text{if } d=3\\
0 &\text{if } n=d > 3\\
\tilde\Delta_b(d-1,n-1) &\text{if } 3 < d < n < 2d\\
\tilde\Delta_b(d-1,n-1)+2 \tilde\Delta_b(d,\lfloor n/2\rfloor)+2  &\text{if } 6 < 2d \leq n.
\end{array}
\right.
\end{equation*}
Since $\Delta_b(3,n) \leq \lfloor\frac{2n}{3}\rfloor-1$, the recursion~\eqref{equation:KalaiKleitman-Deltab} implies $\Delta_b(d,n) \leq \tilde\Delta_b(d,n)$ for $n > d \geq 3$. By focusing on polytopes, a bound similar to Theorem~\ref{theorem:DeltaU-k0} with improved leading coefficient holds.
\begin{theorem}\label{theorem:DeltaB-k0}
If $n > d \geq 3$, then $\Delta_b(d,n) \leq [\frac23(n-d+\frac32)]^{\log (d-1)}$.
\end{theorem}
\begin{proof}
If $d=3$, then by~\cite{KimSantos:HirschSurvey}, $\Delta_b(3,n)  \leq \lfloor\frac{2n}{3}\rfloor-1$ which is at most the stated bound $(\frac23n-1)^{\log(3-1)}$.

The cases when $n < 2d$ follow from monotonicity, so we assume $n \geq 2d$. Suppose $4 \leq d \leq 10$.  Appendix~\ref{appendix:DeltaB-k0} checks the inequality holds when $4 \leq d \leq 10$ and $2d \leq n \leq d+10$.
Since $d \geq 11$ and $2d \leq n \leq d+10$ cannot hold, we have verified all cases when $n \leq d+10$ for all $d$.

We thus assume $n-d \geq 11$ so $\log(\frac23(n-d+\frac32))\geq 3$. Thus,
\begin{align*}
\Delta_u(d,n) &\leq \Delta_u(d-1,n-1) + 2 \cdot \Delta_u(d,\lfloor n/2 \rfloor) + 2 \\
&\leq \textstyle (\frac23(n-d+\frac32))^{\log(d-2)} + 2 \cdot (\frac23(\frac{n}{2}-d+\frac32-1))^{\log (d-1)} + 2\\
&= (d-2)^{\log\frac23(n-d+\frac32)} + 2  \cdot (d-1)^{\log\frac23(\frac{n}{2}-d+\frac32-1)} + 2\\
&=  \left(\frac{d-2}{d-1}\right)^{\log\frac23(n-d+\frac32)} (d-1)^{\log\frac23(n-d+\frac32)} + 2 \cdot (d-1)^{\log\frac23(\frac{n}{2}-d+\frac32-1)} + 2\\
&\leq \left(\frac{d-2}{d-1}\right)^{\log\frac23(n-d+\frac32)} (d-1)^{\log\frac23(n-d+\frac32)} + 2(d-1)^{\log\frac23(n-d+\frac32-1)/2} + 2 \\
&= \left(\frac{d-2}{d-1}\right)^{\log\frac23(n-d+\frac32)} (d-1)^{\log\frac23(n-d+\frac32)} + \frac2{d-1} \cdot (d-1)^{\log\frac23(n-d+\frac32)} + 2\\
&\leq \left(\frac{d-2}{d-1}\right)^{3} (d-1)^{\log\frac23(n-d+\frac32)} + \frac2{d-1} \cdot (d-1)^{\log\frac23(n-d+\frac32)} + 2\\
&= (d-1)^{\log\frac23(n-d+\frac32)} \left[ \left(\frac{d-2}{d-1}\right)^{3}  + \frac2{d-1}   + \frac{2}{(d-1)^{\log\frac23(n-d+\frac32)} } \right]\\
&\leq \textstyle (\frac23(n-d+\frac32))^{\log (d-1)}.
\end{align*}
\end{proof}
The argument can be slightly refined using $\Delta_b(4,9)=\Delta_b(5,10)=5$ from~\cite{Klee:d-step} by Klee, 
$\Delta_b(4,10)=5$ and $\Delta_b(5,11)=6$ in~\cite{Goodey} by Goodey, $\Delta_b(4,11)=\Delta_b(6,12)=6$ in~\cite{Bremner:DiameterFewFacets}, $\Delta_b(4,12)=\Delta_b(5,12)=7$ in~\cite{Bremner:MoreBounds} by Bremner et al., and $\Delta_b(n,d)\leq n-d$ if $n-d \leq 6$ in~\cite{Bremner:DiameterFewFacets}.

\section{Upper bounds for normal simplicial complexes}\label{section:complexes}

In this section, we study the diameters of pure $(d-1)$-dimensional normal simplicial complexes with the pseudomanifold property. Kalai and Kleitman proved that normal simplicial complexes satisfy the Kalai-Kleitman recursion, and this work was extended by Eisenbrand et al.{} (see~\cite{Eisenbrand:Diameter-of-Polyhedra}) to the setting of connected layer families:
\begin{theorem}[\cite{Eisenbrand:Diameter-of-Polyhedra}, \cite{Kalai:DiameterHeight}]
The dual diameter of pure $(d-1)$-dimensional normal simplicial complexes is at most $n^{1+\log(d)}$.
\end{theorem}
In~\cite{Labbe:ExpCombSegments}, Labb\'e, Manneville, and Santos prove the diameter is linear in fixed dimension, the analogue of~\cite{Barnette, Barnette:UpperBound, Larman}.
\begin{theorem}[Labb\'e, Manneville, Santos~\cite{Labbe:ExpCombSegments}]\label{theorem:complex-linear-fixed}
The dual diameter of $(d-1)$-dimensional pure and normal pseudomanifolds without boundary is at most $2^{d-3}n$.
\end{theorem}
Let $\Sigma(d,n)$ denote the maximum dual diameter of pure $(d-1)$-dimensional normal simplicial complexes which are pseudomanifolds without boundary. We recursively define a doubly-indexed sequence $\tilde\Sigma(d,n)$ for $n \geq d \geq 2$ by the following:
\begin{equation*}
\tilde\Sigma(d,n) = \left\{
\begin{array}{ll}
\lfloor\frac{n}{2}\rfloor & \text{if } d=2\\
0 &\text{if } n = d > 2\\
\tilde\Sigma(d-1,n-1) &\text{if } 2 < d \leq n < 2d\\
\tilde\Sigma(d-1,n-1)+2 \tilde\Sigma(d,\lfloor n/2\rfloor)+2  &\text{if } 4 < 2d \leq n.
\end{array}
\right.
\end{equation*}
Since $\Sigma(2,n) \leq \lfloor\frac{n}{2}\rfloor$, the normal property implies $\Sigma(d,n) \leq \tilde\Sigma(d,n)$ for $n \geq d \geq 2$. 
We use the technique of Sukegawa and Kitahara (see~\cite{Sukegawa:KalaiKleitman}) to prove the following.
\begin{theorem}\label{theorem:Sigma}
If $n > d \geq 4$, then $\Sigma(d,n) \leq (n-d)^{\log (d)}$.
\end{theorem}
\begin{proof}
If $d=4$, then Appendix~\ref{appendix:Sigma} shows $\tilde\Sigma(4,n) \leq (n-4)^{\log(4)}$ for $4 < n \leq 15$ and Theorem~\ref{theorem:complex-linear-fixed} implies $\Sigma(4,n) \leq 2^{4-1}n \leq (n-4)^{\log(4)}$ for $n \geq 16$. Suppose $d \geq 5$. The cases when $n < 2d$ follow from monotonicity, so we assume $n \geq 2d$. Appendix~\ref{appendix:Sigma} also checks $\tilde\Sigma(d,n) \leq (n-d)^{\log (d)}$ holds in the cases when $2d \leq n < d+8$. Thus, we assume $d \geq 5$ and $n-d\geq 8$, which implies $\log(n-d) \geq 3$. Then,
\begin{align*}
\Sigma(d,n) &\leq \Sigma(d-1,n-1) + 2 \cdot \Sigma(d,\lfloor n/2 \rfloor) + 2 \\
&\leq \textstyle (n-d)^{\log(d-1)} + 2 \cdot (\frac{n}{2} - d)^{\log (d)} + 2\\
&= (d-1)^{\log(n-d)} + 2  \cdot d^{\log(n/2-d)} + 2\\
&=  \left(\frac{d-1}{d}\right)^{\log(n-d)} d^{\log(n-d)} + 2 \cdot d^{\log(n/2-d)} + 2\\
&\leq \left(\frac{d-1}{d}\right)^{\log(n-d)} d^{\log(n-d)} + 2 \cdot d^{\log((n-d)/2)} + 2 \\
&= \left(\frac{d-1}{d}\right)^{\log(n-d)} d^{\log(n-d)} + \frac2{d} \cdot d^{\log(n-d)} + 2\\
&\leq \left(\frac{d-1}{d}\right)^{3} d^{\log(n-d)} + \frac2{d} \cdot d^{\log(n-d)} + 2\\
&= d^{\log(n-d)} \left[ \left(\frac{d-1}{d}\right)^{3}  + \frac2{d}   + \frac{2}{d^{\log(n-d)}} \right]\\
&\leq d^{\log(n-d)} \left[ \frac{d^3-3d^2+3d-1}{d^3}  + \frac{2d^2}{d^3}   + \frac{2}{d^3} \right]\\
&= d^{\log(n-d)} \cdot \frac{d^3-d^2+3d+1}{d^3}\\
&\leq (n-d)^{\log (d)},
\end{align*}
since $d^3-d^2+3d+1 \leq d^3$ for $d \geq 4$.
\end{proof}

\section{Iterating the Kalai-Kleitman inequality}\label{section:iterated-recursion}

In this section, we prove several intermediate bounds on the diameters of polyhedra which will be needed in Section~\ref{section:tail-polynomial}. In addition to the inequality~\eqref{equation:KalaiKleitman-Deltau} due to Kalai and Kleitman, we will need two other previously-known inequalities.

The first inequality is trivially derived from Proposition 2.9a in~\cite{Klee:d-step}, which states if $n > d > 1$, then $\Delta_u(d,n) \leq \Delta_u(d,n+1)-1$. By reindexing, $\Delta_u(d-1,n-1) \leq \Delta_u(d-1,n)-1$ if $n>d>2$. The inequality also holds in the case $n=d$, because the $(d-1)$-polyhedron $P_{d-1}=\{x \in \mathbb{R}^{d-1}_{\geq 0} \mid \sum_{j=1}^{d-1} x_j \geq 1\}$ with $d$ facets has diameter one. Thus,
\begin{equation}\label{equation:remove-facet}
\Delta_u(d-1,n-1) \leq \Delta_u(d-1,n)-1 \text{ if } n \geq d \geq 3.
\end{equation}

Our second inequality is even more trivial. In the case when $n<2d$, the recursive bound 
\begin{equation}\label{equation:small-facets}
\Delta_u(d,n) \leq \Delta_u(d-1,n-1)
\end{equation}
holds by the observation that any two vertices lie on a common facet. As in~\cite{Ziegler:Lectures}, one may iteratively apply inequality~\eqref{equation:small-facets} until the number of facets is no longer less than twice the dimension. If $n \geq 2d$, one must turn to inequality~\eqref{equation:KalaiKleitman-Deltau}, whose recursive formula is not known. In the following lemma, we analyze the analogous iterative process for $n \geq 2d$. 
\begin{lemma}\label{lemma:iterated-KalaiKleitman}
For $d \geq 3$, $k \geq 0$ integer, and $d\cdot 2^k \leq n < d \cdot 2^{k+1}$,
\begin{equation}\label{equation:iterated-KK-inequality}
\Delta_u(d,n) \leq \sum_{i=0}^k 2^i\Delta_u(d-1,\lfloor  n/{2^i}\rfloor)-1.
\end{equation}
\end{lemma}
\begin{proof}
Let $d \geq 3$ be fixed. We proceed by induction on $k$.  If $k=0$, then $d \leq n < 2d$, so by~\eqref{equation:remove-facet} and~\eqref{equation:small-facets}, $\Delta_u(d,n) \leq \Delta_u(d-1,n-1) \leq \Delta_u(d-1,n)-1$, which is the right side of~\eqref{equation:iterated-KK-inequality}. Assume that~\eqref{equation:iterated-KK-inequality} holds for $k \geq 0$. Let $d\cdot 2^{k+1}\leq n < d \cdot 2^{k+2}$. Then $d \cdot 2^k \leq \lfloor n/2 \rfloor < d \cdot 2^{k+1}$. By our hypothesis and~\eqref{equation:KalaiKleitman-Deltau},
\begin{align*}
\Delta_u(d,n) &\leq \Delta_u(d-1,n-1)+2\Delta_u(d,\lfloor n/2\rfloor)+2\\
&\leq \Delta_u(d-1,n-1)+2\left[\sum_{i=0}^k 2^i \Delta_u(d-1,\lfloor  n/{2^{i+1}}\rfloor)-1\right]+2 \\
&\leq \Delta_u(d-1,n)-1+\sum_{i=0}^k 2^{i+1}\Delta_u(d-1,\lfloor  n/{2^{i+1}}\rfloor)-2+2 \\
&= \Delta_u(d-1,n)+\sum_{i=1}^{k+1} 2^i\Delta_u(d-1,\lfloor  n/{2^i}\rfloor)-1 \\
&= \sum_{i=0}^{k+1} 2^i\Delta_u(d-1,\lfloor  n/{2^i}\rfloor)-1.
\end{align*}
Therefore, the resulting bound \eqref{equation:iterated-KK-inequality} holds by induction.
\end{proof}
For $d=4$, this bound aligns with the upper bound $\tilde{\Delta}_u(4,n)$ mentioned in Section~\ref{section:polyhedra}. The Kalai-Kleitman recursion has had the following long-standing geometric interpretation: Given two arbitrary vertices $u$ and $v$ in the graph $G$ of a polyhedron $P$, there is a subgraph $G_u$ of $G$ incident to $\frac{n}{2}$ facets, a subgraph $G_v$ of $G$ incident to $\frac{n}{2}$ facets and a common facet $F$ containing vertices $u'$ and $v'$ incident to $G_u$ and $G_v$, respectively. The efficiency of the bound on $\Delta_u(d,n)$ in the previous lemma comes from applying the same inequality inductively for the paths from $u$ to $u'$ and from $v$ to $v'$.

\begin{theorem}\label{theorem:sum-binary-facets}
For $n \geq d \geq 3$,
\[\Delta_u(d,n) \leq \sum_{i=0}^{\lfloor \log n/d\rfloor} 2^i \Delta_u(d-1,\lfloor  n/2^i\rfloor)-1.\]
\end{theorem}
\begin{proof}
Note that $d \cdot 2^k \leq n$ implies $k \leq \lfloor\log n/d \rfloor$.
\end{proof}

\begin{lemma}\label{lemma:index-swap}
For $k \geq 0$,
\[
\sum_{i_1=0}^k\sum_{i_2=0}^{k-i_1}\cdots \sum_{i_p=0}^{k-i_1-\dots-i_{p-1}}1
=
\sum_{i_1=0}^k\sum_{i_2=0}^{i_1}\cdots \sum_{i_p=0}^{i_{p-1}} 1
=
\binom{k+p}k
\]
\end{lemma}
\begin{proof}The first equality follows from induction on $p$ and the base case
\[
\sum_{i_1=0}^k\sum_{i_2=0}^{k-i_1}1=\sum_{i_1=0}^k\sum_{i_2=0}^{i_1} 1,
\]
which in turn can be shown by induction on $k$, or by expanding and regrouping. The second equality is found in~\cite{butler}.
\end{proof}

The following theorem is a refinement of the statement $\Delta_u(d,2^k)\leq 2^k \binom{k+d}{d}=2^k\binom{k+d}{k}$ from~\cite{KimSantos:HirschSurveyCompanion}.
\begin{theorem}\label{theorem:binomial-bound}
For $n \geq d \geq 3$,
\[
\Delta_u(d,n) \leq (n-3)\binom{\lfloor \log \frac n{4}\rfloor+d-3}{\lfloor \log \frac n{4}\rfloor}.
\]
\end{theorem}
\begin{proof} Let $d \geq 4$. Our aim is to recursively apply Theorem~\ref{theorem:sum-binary-facets} until the first argument of $\Delta_u$ inside the nested summations is equal to 3, at which point we apply $\Delta_u(3,n)\leq n-3$ from~\cite{Klee:PathsII}. We first apply Theorem~\ref{theorem:sum-binary-facets} twice to illustrate how we intend to simplify:
\begin{align}
\nonumber \Delta_u(d,n) &\leq \sum_{i_1=0}^{\lfloor \log \frac nd\rfloor} 2^{i_1} \Delta_u\left(d-1,\left\lfloor  \frac n{2^{i_1}}\right\rfloor\right)-1\\
\nonumber &\leq \sum_{i_1=0}^{\lfloor \log \frac nd\rfloor} 2^{i_1} \Delta_u\left(d-1,\left\lfloor  \frac n{2^{i_1}}\right\rfloor\right)\\
\nonumber &\leq \sum_{i_1=0}^{\lfloor \log \frac nd\rfloor} 2^{i_1} \left[\sum_{i_2=0}^{\left\lfloor \log \frac{\lfloor n /2^{i_1}\rfloor}{d-1}\right\rfloor} 2^{i_2}\Delta_u\left(d-2,\left\lfloor \frac n{2^{i_1+i_2}}\right\rfloor\right)-1 \right]\\
\nonumber &\leq \sum_{i_1=0}^{\lfloor \log \frac nd\rfloor} \sum_{i_2=0}^{\left\lfloor \log \frac{n}{d-1}\right\rfloor-i_1} 2^{i_1+i_2}\Delta_u\left(d-2,\left\lfloor \frac n{2^{i_1+i_2}}\right\rfloor\right).
\end{align}
Continuing this recursion until we have applied Theorem \ref{theorem:sum-binary-facets} exactly $d-3$ times, we obtain
\begin{align*}
\Delta_u(d,n) &\leq\sum_{i_1=0}^{\lfloor \log \frac nd\rfloor} \sum_{i_2=0}^{\lfloor \log \frac n{d-1}\rfloor-i_1}\cdots \sum_{i_{d-3}=0}^{\lfloor \log \frac n{4}\rfloor-i_1-i_2-\dots-i_{d-4}} 2^{i_1+i_2+\dots+i_{d-3}}\Delta_u\left(3,\left\lfloor  \frac n{2^{i_1+i_2+\dots+i_{d-3}}}\right\rfloor\right).
\intertext{Applying the bound $\Delta_u(3,n)\leq n-3$ from~\cite{Klee:PathsII}, the above expression is bounded by}
&\leq \sum_{i_1=0}^{\lfloor \log \frac nd\rfloor} \sum_{i_2=0}^{\lfloor \log \frac n{d-1}\rfloor-i_1}\cdots \sum_{i_{d-3}=0}^{\lfloor \log \frac n{4}\rfloor-i_1-i_2-\dots-i_{d-4}} 2^{i_1+i_2+\dots+i_{d-3}} \left(\left\lfloor  \frac n{2^{i_1+i_2+\dots+i_{d-3}}}\right\rfloor-3\right)\\
&\leq \sum_{i_1=0}^{\lfloor \log \frac nd\rfloor} \sum_{i_2=0}^{\lfloor \log \frac n{d-1}\rfloor-i_1}\cdots \sum_{i_{d-3}=0}^{\lfloor \log \frac n{4}\rfloor-i_1-i_2-\dots-i_{d-4}} (n-3\cdot2^{i_1+i_2+\dots+i_{d-3}})\\
&\leq \sum_{i_1=0}^{\lfloor \log \frac nd\rfloor} \sum_{i_2=0}^{\lfloor \log \frac n{d-1}\rfloor-i_1}\cdots \sum_{i_{d-3}=0}^{\lfloor \log \frac n{4}\rfloor-i_1-i_2-\dots-i_{d-4}} (n-3)\\
&=(n-3) \sum_{i_1=0}^{\lfloor \log \frac nd\rfloor} \sum_{i_2=0}^{\lfloor \log \frac n{d-1}\rfloor-i_1}\cdots \sum_{i_{d-3}=0}^{\lfloor \log \frac n{4}\rfloor-i_1-i_2-\dots-i_{d-4}} 1\\
&\leq (n-3)\sum_{i_1=0}^{\lfloor \log \frac n{4}\rfloor}\sum_{i_2=0}^{\lfloor \log \frac n{4}\rfloor-i_1}\cdots\sum_{i_{d-3}=0}^{\lfloor \log \frac n{4}\rfloor-i_1-\dots-i_{d-3}}1\\
&= (n-3)\sum_{i_1=0}^{\lfloor \log \frac n{4}\rfloor}\sum_{i_2=0}^{i_1}\cdots\sum_{i_{d-3}=0}^{i_{d-3}}1\\
&= (n-3)\binom{\lfloor \log \frac n{4}\rfloor+d-3}{\lfloor \log \frac n{4}\rfloor},
\end{align*}
where the last two equalities use Lemma~\ref{lemma:index-swap}. 
\end{proof}
A similar bound of
\[
\Delta_b(d,n) \leq (\lfloor \tfrac{2n}3\rfloor-1)\binom{\lfloor \log \frac n{4}\rfloor+d-3}{\lfloor \log \frac n{4}\rfloor}
\]
holds for polytopes. We note that similar binomial bounds hold for subset partition graphs with dimension reduction and for normal simplicial complexes with the pseudomanifold property.

\section{Tail-polynomial bounds for polyhedra}\label{section:tail-polynomial}

In this section, we present successively sharper tail-polynomial upper bounds on the diameters of convex polyhedra. Our first bound is cubic in the number $n$ of facets.
\begin{theorem}\label{theorem:cubic}
For $d \geq 3$ and $n \geq 2^{d-1}$,
\[
\Delta_u(d,n) \leq \frac1{16}\frac{n^3}{\sqrt{3\log n-5}}.
\]
\end{theorem}
\begin{proof} Let $d \geq 3$ and $n \geq 2^{d-1}$ so that $d-3 \leq \lfloor \log n/4 \rfloor$. By Theorem \ref{theorem:binomial-bound} and Stirling's approximation,
\begin{align*}
\Delta_u(d,n) 
&\leq (n-3)\binom{\lfloor \log \frac n{4}\rfloor+d-3}{\lfloor \log \frac n{4}\rfloor}\\
&\leq (n-3)\binom{2\lfloor \log \frac n{4}\rfloor}{\lfloor \log \frac n{4}\rfloor}\\
&\leq (n-3)\frac{4^{\lfloor \log \frac n{4}\rfloor}}{\sqrt{3\lfloor \log \frac n{4}\rfloor+1}}\\
&\leq (n-3)\frac{(n/4)^2}{\sqrt{3\log n-5}}\\[.5em]
&\leq \frac1{16}\frac{n^3}{\sqrt{3\log n-5}}.
\end{align*}
\end{proof}

This result improves to a tail-subcubic bound by using a direct upper bound on the binomial coefficient from Theorem~\ref{theorem:binomial-bound} instead of using the central binomial coefficient.
\begin{theorem}\label{theorem:subcubic}
Let $\varepsilon > 0$ and $d \geq 3$. For sufficiently large $n$,
\[
\Delta_u(d,n) \leq n^{1+\frac{1}{\ln2}+\varepsilon}.
\]
\end{theorem}
\begin{proof}
Let $d \geq 3$.
Suppose $n \geq 2^{1+\frac{d-3}{2^\varepsilon-1}}$. Then $2^\varepsilon > 0$ and $d-3 \leq (2^\varepsilon-1)\lfloor{\log n}\rfloor$. Using Theorem~\ref{theorem:binomial-bound},
\begin{align*}
\Delta_u(d,n)
&\leq (n-3)\binom{\lfloor \log \frac n{4}\rfloor+d-3}{\lfloor \log \frac n{4}\rfloor}\\
&\leq (n-3)\binom{\lfloor \log n\rfloor+d-3}{\lfloor \log n \rfloor} \\
&< n\left( \frac{e(\lfloor \log n \rfloor + d-3)}{\lfloor \log n \rfloor} \right)^{\lfloor \log n \rfloor} \\
&\leq n\left( \frac{e(\lfloor \log n \rfloor + (2^\varepsilon-1)\lfloor{\log n}\rfloor)}{\lfloor \log n \rfloor} \right)^{\lfloor \log n \rfloor} \\
&=n(2^\varepsilon e)^{\lfloor \log n \rfloor} \\
&\leq n2^{\varepsilon \log n} e^{\log n} \\
&=n^{1+\varepsilon} \cdot n^{\frac{1}{\ln 2}}.
\end{align*}
\end{proof}

If $n$ is much larger than $d$, then the binomial coefficient $\binom{\lfloor \log n \rfloor + d - 3}{\lfloor \log n \rfloor}$ is found near the boundary of Pascal's triangle. The following lemma shows that this binomial coefficient is small (relative to $n$) if $d-3$ is much smaller than $n$.
\begin{lemma}
Let $\varepsilon > 0$ and let $d \geq 3$ be given. For sufficiently large $n$,
\[
\binom{\lfloor \log n \rfloor + d - 3}{\lfloor \log n \rfloor} \leq n^\varepsilon.
\]
\end{lemma}
\begin{proof}
Fix $\varepsilon > 0$ and let $s = \max\{\frac{64}{\varepsilon^2},e\}$. Then
$\frac{\log s}{s} < \frac{\sqrt{s}}{s} = \frac1{\sqrt{s}} = \frac{\varepsilon}{8}$,
so
\[\frac{4\log (es)}{s} = \frac{4(\log e + \log s)}{s} \leq \frac{8\log s}{s} < \varepsilon.\]
So, $\frac{2\lfloor \log n \rfloor}{s} \log(es) < \frac{\varepsilon}{2} \lfloor \log n \rfloor$. Then $\log((es)^{2\lfloor \log n \rfloor/s}) < \frac{\varepsilon}{2} \lfloor \log n \rfloor \leq \log n^{\varepsilon/2}$. Therefore,
\begin{equation}\label{equation:n-epsilon-over-2}
(es)^{2\lfloor \log n \rfloor/s} \leq n^{\varepsilon/2}.
\end{equation}
Now suppose $n \geq 2^{32d/\varepsilon^2}$. Since $n \geq 2^{s/2}$, hence $\frac{2 \lfloor \log n \rfloor}{s} \geq 1$, so by~\eqref{equation:n-epsilon-over-2},
we have $es \leq n^{\varepsilon/2}$, which together with~\eqref{equation:n-epsilon-over-2} implies
\begin{equation}\label{equation:n-epsilon}
(es)^{2\lfloor \log n \rfloor/s+1} \leq n^\varepsilon.
\end{equation}
Since $n \geq 2^{32d/\varepsilon^2}$, we have $n \geq 2^{ds/2} \geq 2^{\frac{s}{2}(d-3)-1}$ so $\frac{s}{2}(d-3)-1 \leq \log n$, which implies $\frac{s}{2}(d-3) \leq \lfloor \log n \rfloor$, and thus $d-3 \leq \lceil \frac{2}{s}\lfloor \log n \rfloor \rceil$. Further, $d-3 \leq \lfloor \log n \rfloor$ since $s > 2$.
\begin{align*}
\binom{\lfloor \log n \rfloor + d - 3}{\lfloor \log n \rfloor}
&= \binom{\lfloor \log n \rfloor + d - 3}{d-3} \\
&\leq \binom{2\lfloor \log n \rfloor}{d-3} \\
&\leq \binom{2\lfloor \log n \rfloor}{\lceil \frac2s\lfloor \log n \rfloor \rceil} \\
&< \left(\frac{e \cdot 2\lfloor \log n \rfloor}{\lceil\frac2s\lfloor \log n \rfloor\rceil}\right)^{\lceil\frac2s\lfloor\log n\rfloor\rceil} \\
&\leq \left(\frac{e \cdot 2\lfloor \log n \rfloor}{\frac2s\lfloor \log n \rfloor}\right)^{1+\frac2s\lfloor\log n\rfloor} \\
&= (es)^{1+2\lfloor \log n \rfloor/s},
\end{align*}
where the first inequality follows from $d-3 \leq \lfloor \log n \rfloor$ and the second inequality comes from $d-3 \leq \lceil \frac{2}{s}\lfloor \log n \rfloor \rceil$. By~\eqref{equation:n-epsilon}, the last expression is bounded above by $n^\varepsilon$.  
\end{proof}

The previous lemma implies the following tail-almost-linear diameter upper bound for polyhedra.
\begin{theorem}\label{theorem:almost-linear}
For $\varepsilon > 0$ and $d \geq 3$, if $n \geq 2^{32d/\varepsilon^2}$, then
\[\Delta_u(d,n) \leq n^{1+\varepsilon}.\]
\end{theorem}
\begin{proof}
Fix $\varepsilon > 0$ and $d \geq 3$. Let $n \geq 2^{32d/\varepsilon^2}$. By the previous lemma and Theorem~\ref{theorem:binomial-bound},
\begin{align*}
\Delta_u(d,n)
&\leq (n-3)\binom{\lfloor \log \frac n{4}\rfloor+d-3}{\lfloor \log \frac n{4}\rfloor}\\
&\leq n\binom{\lfloor \log n\rfloor+d-3}{\lfloor \log n \rfloor} \\
&\leq n^{1+\varepsilon}.
\end{align*}
\end{proof}

\begin{corollary}\label{corollary:almost-linear}
For $\varepsilon > 0$, $d \geq 3$, and $n \geq 2^{32d/\varepsilon^2}$, one has 
$\Delta_b(d,n) \leq \Delta_u(d,n) \leq n^{1+\varepsilon}$.
\end{corollary}

\section{Concluding remarks}\label{section:concluding-remarks}

It is not likely that one can obtain better than tail-polynomial bounds for diameters of polyhedra using just the Kalai-Kleitman inequality. (See the discussion in~\cite{Ziegler:Lectures}).

We note that the results of Theorems~\ref{theorem:cubic}, \ref{theorem:subcubic}, and~\ref{theorem:almost-linear} are not proofs of the Polynomial Hirsch Conjecture, which asserts there is a polynomial $g(d,n)$ in $d$ and $n$ such that for all $n \geq d$, one has $\Delta_u(d,n) \leq g(d,n)$. However, we emphasize that these are tail results in the sense that for each $d$, the graph of the function $n \mapsto \Delta_u(d,n)$ is eventually below the graph of $g(n)=n^{1+\varepsilon}$, which gives very strong evidence for the Polynomial Hirsch Conjecture, and even the Linear Hirsch Conjecture, though an inequality much stronger than~\eqref{equation:remove-facet} is needed to apply extrapolation to obtain good bounds for $\Delta_u(d,n)$ when $n$ is small. Especially in light of Theorem~\ref{theorem:almost-linear}, it would be very surprising if the Polynomial Hirsch Conjecture were false.

The bounds in Theorems~\ref{theorem:cubic}, \ref{theorem:subcubic}, and~\ref{theorem:almost-linear} have expressions which are uniform in $d$, though the initial value for which the tail bound applies is exponential in $d$. These results may demonstrate the limitations of using bounds based on binomial coefficients. Perhaps a refinement of these arguments may lead to an affirmation of H\"ahnle's conjecture (see~\cite{Hahnle:Diplomathesis}) asserting $\Delta_u(n,d) \leq d(n-1)$. Given Theorem~\ref{theorem:almost-linear} and~\eqref{equation:remove-facet}, we conjecture the slightly weaker statement that $\Delta_u(n,d)$ is bounded above by $n^{2+\varepsilon}$.

Though linear programming can be solved in polynomial time using Kachiyan's ellipsoid method (see~\cite{Khachiyan}) and Karmarkar's interior point method (see~\cite{Karmarkar}), linear programming is not known to be strongly polynomial time in the sense of~\cite{BCSS}. Further details are found in~\cite{Grotschel:Geometric-Algorithms}, \cite{Megiddo:ComplexityLP}, and~\cite{Schrijver:CombinatorialOptimization-Polyhedra-Efficiency}. Even if the Polynomial Hirsch Conjecture is true, the bigger question remains the existence of a pivot rule so that the simplex algorithm for linear programming applies a polynomial number of simplex pivots. The existence of such a rule would admit a strongly polynomial time algorithm for linear programming, answering Smale's (see~\cite{Smale}) ninth problem.

\section*{Acknowledgments}

Research for both authors was partially supported by a Faculty Research Grant at the University of Wisconsin-La Crosse.

\newpage

\appendix

\section{Sporadic cases}

This appendix includes {\tt SAGE} code~\cite{sage} which verifies inequalities in sporadic cases.

\subsection{Sporadic cases for Theorem~\ref{theorem:DeltaU-k0}}
\label{appendix:DeltaU-k0}
 
The following {\tt SAGE} code verifies that $\tilde\Delta_u(d,n) \leq (n-d)^{\log (d-1)}$ for $d=4 \leq n \leq 45$, or $2d \leq n < d+8$.

\begin{verbatim}
upper_bound(d,n)=(n-d)^log(d-1,2)
max_n = 45
max_d = 7
Tilde = {}

for d in range(3,max_d+1):
    for n in range(3,max_n+1):
        if d==3:
            Tilde[d,n]=n-d
        if d>=4:
            if n==d:
                Tilde[d,n] = 0
            if n>=d and n<2*d:
                Tilde[d,n] = Tilde[d-1,n-1]
            if n>=2*d:
                Tilde[d,n] = Tilde[d-1,n-1] + 2*Tilde[d,floor(n/2)] + 2

for d in range(4,max_d+1):
    for n in range(4,max_n+1):
        if d==4:
            if Tilde[d,n] <= upper_bound(d,n):
                print "Passed " + str(d) + ", " + str(n)
            else:
                print "Failed " + str(d) + ", " + str(n)
        if d>=5:
            if (n>=2*d and n<d+8):
                if Tilde[d,n] <= upper_bound(d,n):
                    print "Passed " + str(d) + ", " + str(n)
                else:
                    print "Failed " + str(d) + ", " + str(n)
\end{verbatim}

\subsection{Sporadic cases for Theorem~\ref{theorem:DeltaB-k0}}\label{appendix:DeltaB-k0}

The following {\tt SAGE} code verifies that $\tilde{\Delta}_b(d,n) \leq (\frac23(n-d+\frac32))^{\log(d-1)}$ for $d \geq 4$ or $2d \leq n < d+11$.

\begin{verbatim}
upper_bound(d,n)=((2/3)*(n-d+3/2))^log(d-1,2)
max_n = 20
max_d = 10
Tilde = {}

for d in range(3,max_d+1):
    for n in range(3,max_n+1):
        if d==3:
            Tilde[d,n]=floor(2*n/3)-1
        if d>=4:
            if n>=d and n<2*d:
                Tilde[d,n] = Tilde[d-1,n-1]
            if n>=2*d:
                Tilde[d,n] = Tilde[d-1,n-1] + 2*Tilde[d,floor(n/2)] + 2

for d in range(4,max_d+1):
    for n in range(8,max_n+1):
        if (n>=2*d and n<d+11):
            if Tilde[d,n] <= upper_bound(d,n):
                print "Passed " + str(d) + ", " + str(n)
            else:
                print "Failed " + str(d) + ", " + str(n)
\end{verbatim}

\subsection{Sporadic cases for Theorem~\ref{theorem:Sigma}}\label{appendix:Sigma}

The following {\tt SAGE} code verifies that $\tilde\Sigma(d,n) \leq (n-d)^{\log (d)}$ for $d = 4 < n \leq 15$ or $2d \leq n < d+8$.

\begin{verbatim}
upper_bound(d,n)=(n-d)^log(d,2)
max_n = 15
max_d = 7
Tilde = {}

for d in range(2,max_d+1):
    for n in range(2,max_n+1):
        if d==2:
            Tilde[d,n]=floor(n/2)
        if d>=3:
            if n>=d and n<2*d:
                Tilde[d,n] = Tilde[d-1,n-1]
            if n>=2*d:
                Tilde[d,n] = Tilde[d-1,n-1] + 2*Tilde[d,floor(n/2)] + 2

for d in range(4,max_d+1):
    for n in range(5,max_n+1):
        if d==4:
            if Tilde[d,n] <= upper_bound(d,n):
                print "Passed " + str(d) + ", " + str(n)
            else:
                print "Failed " + str(d) + ", " + str(n)
        if d>=5:
            if (n>=2*d and n<d+8):
                if Tilde[d,n] <= upper_bound(d,n):
                    print "Passed " + str(d) + ", " + str(n)
                else:
                    print "Failed " + str(d) + ", " + str(n)
\end{verbatim}

\end{document}